\documentclass[12pt,a4paper]{amsart}

\usepackage[a4paper,margin=20mm,top=25mm]{geometry}

\usepackage{hyperref}
\usepackage[utf8]{inputenc} 

\usepackage{amsfonts, amsmath, amssymb, amsgen, amsthm}
\usepackage{newtxtext, newtxmath}

\usepackage[all]{xy}
\usepackage{enumerate}
\usepackage{xcolor,listings,algpseudocode}
\newcommand{\C}[1]{\mathcal{#1}}

\renewcommand{\leq}{\leqslant}
\renewcommand{\geq}{\geqslant}

\numberwithin{equation}{section}

\newtheorem{theorem}{Theorem}[section]
\newtheorem{proposition}[theorem]{Proposition}

\title{Enumerating Labeled Graphs that Realize a Fixed Degree Sequence}

\author{Atabey Kaygun}

\address{Istanbul Technical University, Istanbul, Turkey.}
\email{kaygun@itu.edu.tr}

\begin{document}
\maketitle

\begin{abstract}
  A finite non-increasing sequence of positive integers
  $d = (d_1\geq \cdots\geq d_n)$ is called a \emph{degree sequence} if
  there is a graph $G = (V,E)$ with $V = \{v_1,\ldots,v_n\}$ and
  $\deg(v_i)=d_i$ for $i=1,\ldots,n$.  In that case we say that the
  graph $G$ \emph{realizes the degree sequence} $d$.  We show that the
  exact number of labeled graphs that realize a fixed degree sequence
  satisfies a simple recurrence relation.  Using this relation, we
  then obtain a recursive algorithm for the exact count.  We also show
  that in the case of regular graphs the complexity of our algorithm
  is better than the complexity of the same enumeration that uses
  generating functions.
\end{abstract}

\section*{Introduction}

A finite non-increasing sequence of positive integers
$d_1\geq \cdots\geq d_n$ is called a \emph{degree sequence} if there
is a graph $(V,E)$ with $V = \{v_1,\ldots,v_n\}$ and $\deg(v_i)=d_i$
for $i=1,\ldots,n$. In that case, we say that the graph $G$
\emph{realizes the degree sequence} $d$.  In this article, in
Theorem~\ref{thm:main} we give a remarkably simple recurrence relation
for the exact number of labeled graphs that realize a fixed degree
sequence $(d_1,\ldots,d_n)$.  We also give an algorithm and a concrete
implementation to explicitly count three classes of labeled graphs for
a moderate number of vertices.

There is an extensive volume of research on the asymptotics of the
number of graphs that realize a fixed degree sequence. We refer the
reader to Wormald's excellent ICM lecture~\cite{Wormald18} for an
comprehensive survey, and references therein.  As for the exact number
of graphs that realize a fixed degree sequence, Read obtains
enumeration formulas in~\cite{Read59} and \cite{Read60} as
applications of Polya's
\emph{Hauptsatz}~\cite{Polya39,PolyaRead87}. However, Read himself
admits
\begin{quote}
  ``It may readily be seen that to evaluate the above expressions in
  particular cases may involve an inordinate amount of
  computation.''~\cite[Sect.7]{Read60}
\end{quote}
But we encountered no explicit complexity analysis of Read's formulas
in our search in the literature.  On the other hand, in~\cite{McKay83}
McKay writes explicit generating polynomials whose complexities can
readily be calculated, and in which coefficients of certain monomials
yield the exact number of different classes of labeled graphs. In
particular, he writes a generating polynomial (see
Equation~\eqref{eq:McKay}), whose computation complexity is
$\C{O}(2^{n^2/2})$, in which the coefficient of the monomial
$x_1^m\cdots x_n^m$ gives the exact count of labeled $m$-regular
graphs on $n$-vertices.

Our recurrence relation works for all degree sequences, but for those
degree sequences where there is a uniform upper limit $m$ for the
degrees, our algorithm has the worst-case complexity of
$\C{O}(n^{mn})$.  This means, in the specific case of $m$-regular
graphs we achieve a better complexity than generating polynomials.

While factorial-like worst-case complexity of the
enumeration~\eqref{eq:enumeration} may render practical calculations
difficult, the fact that it is recursive allows us to employ
computational tactics such as \emph{dynamic
  programming}~\cite{Bellman57,LewMauch07} or
\emph{memoization}~\cite{Michie68} to achieve better average
complexity.  We explore this avenue in our implementation given in the
Appendix.  To demonstrate of the versatility of our recurrence
relation and the resulting implementation, we tabulate the number of
$m$-regular labeled graphs, the number of labeled graphs that realize
the same degree sequence with binary trees, and the number of labeled
graphs that realize the same degree sequence with complete bipartite
graphs.

One can also read a given degree sequence $(d_1,\ldots,d_n)$ as a
partition of $N = \sum_i d_i$.  The Erdős-Gallai
Theorem~\cite{ErdosGallai, Chodum:ErdosGallai} tells us when such a
partition is realizable as a degree sequence, or one can also use
Havel-Hakimi algorithm to decide whether the given partition is
realizable~\cite{Havel55,Hakimi62}.  Now, one can also use our
enumeration to decide whether given a degree sequence is realizable,
but admittedly, the Havel-Hakimi algorithm has a much better
complexity.

% Finally, to clear any possible confusion, we would like to mention
% here that there is a whole orthogonal line of research to our work
% in which instead of counting the set of graphs that realize a fixed
% degree sequence as we do, one can enumerate all degree sequences
% among all integer partitions of a fixed
% integer~\cite{BarnesSavage95a,BarnesSavage97,Kohnert04,Wang19}.

\subsection*{Plan of the article} We prove our recurrence relation,
analyze its complexity and compare it with generating functions for
regular graphs in Section~\ref{Sect:Enumeration}.  We present explicit
calculations we made in Section~\ref{Sect:Calculations}, and the code
we used performing the calculations in the Appendix.

\subsection*{Notations and conventions}

We assume all graphs are simple, labeled, and undirected throughout
the article.

\subsection*{Acknowledgments}

This article was written while the author was on academic leave at
Queen’s University in Canada from Istanbul Technical University. The
author would like to thank both universities for their support.

\section{Enumerating Graphs That Realize a Fixed Degree Sequence}
\label{Sect:Enumeration}

\subsection{The recurrence relation}

Assume $d=(d_1,\ldots,d_n)$ is a non-increasing sequence of strictly
positive integers $d_i>0$.  Let us consider the trivial cases first:
It is clear that there is a single graph on the empty sequence
$\epsilon$: the empty graph.  Also, in case $n=1$, the only case for
which the sequence $(d_1)$ is realized by a graph is when $d_1=0$
which is excluded by our assumption. So, the count is 0 for all
$(d_1)$ for $d_1>0$.  We also exclude the case where the sum
$\sum_i d_i$ is odd since such sequences cannot be realized as degree
sequences because of the Hand-Shake Lemma.

We assume $n>1$ and the sum $\sum_i d_i$ is even.  If we consider the
vertex $x_n$ we see that it needs to be connected to exactly $d_n$
vertices in the set $\{x_1,\ldots,x_{n-1}\}$.  We need to consider the
set of all subsets of $\{x_1,\ldots,x_{n-1}\}$ of size $d_n$ to
enumerate all possibilities.  So, let $S$ be an arbitrary subset of
$\{1,\ldots,n-1\}$ of size $d_n$, and let $\chi_S$ be the
characteristic function of the set $S$.  Every graph in which $x_n$ is
connected to each vertex in $\{x_i\mid i\in S\}$ realizes the same
degree sequence
\begin{equation}
  \label{eq:1}
  (d_1 - \chi_S(1),\ldots, d_{n-1} - \chi_S(n-1))
\end{equation}
if we remove $x_n$ and all the edges connected to $x_n$.  Let us write
$(d_1,\ldots,d_{n-1})\slash S$ for the sequence~\eqref{eq:1} after we
reorder the sequence in descending order and remove all 0's.  Let
$C((d_1,\ldots,d_n))$ be the number of graphs that realize the same
degree sequence $(d_1,\ldots,d_n)$.  Thus we obtain:
\begin{theorem}\label{thm:main}
  The total number of labeled graphs that realize the degree sequence
  $(d_1,\ldots,d_n)$ satisfies the recurrence relation
  \begin{equation}
    \label{eq:enumeration}
    C((d_1,\ldots,d_n)) = \sum_{S\in\binom{\{1,\ldots,n-1\}}{d_n}} C((d_1,\ldots,d_{n-1})\slash S)
  \end{equation}
  where we write $\binom{X}{k}$ for the set of all subsets of size $k$
  of a set $X$.
\end{theorem}

\subsection{The complexity analysis}

In~\cite{McKay83} McKay writes a generating polynomial
\begin{equation}
  \label{eq:McKay}
  f(x) = \prod_{1\leq i<j\leq n} (1 + x_i x_j). 
\end{equation}
in which the coefficient of the monomial $x_1^{m}\cdots x_n^{m}$
yields the number of $m$-regular graphs on $n$-vertices.  If we assume
the complexity of the calculation is given by the number of
multiplications in the product, then the computational complexity of
the generating polynomial is $\C{O}(2^{n^2/2})$.

Let $d=(d_1,\ldots,d_n)$ be a degree sequence, and let us use
$\# C(d)$ for the total number of summands (which is the number of
leaves in the recursion tree) in $C(d)$ in
Equation~\eqref{eq:enumeration} which will be the complexity measure
for our enumeration.

\begin{proposition}\label{prop:complexity}
  Assume there is a fixed upper bound $m$ for the degrees in $d$.
  Then the complexity of the enumeration given in~\eqref{thm:main} is
  $\C{O}(n^{mn})$.  In particular, the enumeration complexity for
  $m$-regular graphs is also $\C{O}(n^{mn})$.
\end{proposition}

\begin{proof}
  As long as $2m\leq n$, the function $\binom{n}{m}$ is increasing in
  $m$.  Then
  \begin{align}
    \label{eq:10}
    \# C((d_1,\ldots,d_n))
    = & \sum_{S\in\binom{\{1,\ldots,n-1\}}{d_n}}\# C((d_1,\ldots,d_{n-1})\slash S)\\
    \leq & \binom{n}{m} \max_{S\in 2^{\{1,\ldots,n-1\}}}\# C((d_1,\ldots,d_{n-1})\slash S)\\
    \leq & \cdots \leq \binom{n}{m}\cdots \binom{2m}{m} \max_{S\in 2^{\{1,\ldots,2m-1\}}} \# C((d_1,\dots,d_{2m-1})\slash S) \\
    \leq & \binom{n}{m}^{n-2m} C_m
  \end{align}
  for some constant $C_m$.  Since $m$ is fixed and $\binom{n}{m}$ is
  of order $n^m$ we get that the number of summands in
  $C((d_1,\ldots,d_n))$ is $\C{O}(n^{m(n-2m)}) = \C{O}(n^{mn})$.
\end{proof}

One can easily see that the enumeration complexity
of~\eqref{eq:enumeration} we obtained in
Proposition~\ref{prop:complexity} is better that the complexity of
generating polynomial for regular graphs.  However, we still have to
work around the fact that the worst-case complexity is factorial-like
with a constant exponent.  Fortunately, one can employ powerful
computational tactics such as dynamic programming or memoization to
improve average complexity of the enumeration since it is
recursive. See the Appendix for how we used memoization to improve
average complexity of our calculations.

\section{Explicit Calculations}
\label{Sect:Calculations}

Let us start with calculating an explicit example by hand. The degree
sequence of the complete graph $K_n$ on $n$-vertices is the constant
sequence $n-1$ of length $n$.  Since one has only one subset of
$\{1,\ldots,n-1\}$ of size $n-1$ we get that
\begin{align}\label{eq:13}
  C((\underbrace{n-1,\ldots,n-1}_\text{$n$-times}))
  = C((\underbrace{n-2,\ldots,n-2}_\text{$n-1$-times}))
  = \cdots = C((1,1)) = C(\epsilon) = 1.
\end{align}
In other words, the labeled complete graph $K_n$ is the only graph
with that specific degree sequence.

\subsection{Enumerating regular graphs}

An $m$-regular graph on $n$-vertices is similar to $K_{m+1}$ in that
it is a graph on $n$-vertices where every vertex has the same constant
degree $m$
\begin{equation}
  \label{eq:3}
   (\underbrace{m,\ldots,m}_{\text{$n$-times}}).
\end{equation}
Now, let us write
\begin{equation}
  \label{eq:4}
  R(n,m) = C((\underbrace{m,\ldots,m}_{\text{$m$-times}})).
\end{equation}
We need to note that $R(n,m)=0$ when $m\geq n$, or when both $n$ and
$m$ are odd since in these cases there are no graphs that can realize
the sequences given in~\eqref{eq:3}.

We calculated $R(n,m)$ for $1\leq n\leq 30$ and $2\leq m\leq 8$.  The
results for $1\leq m\leq 5$ took about 2 minutes while the cases for
$6\leq m\leq 8$ took about 10 minutes on a moderate
computer\footnote{On an Intel i5-8250U CPU working at 1.60GHz with 8Gb
  of RAM on a Linux operating system.}.  These calculations strongly
indicate the average complexity of the enumeration algorithm with
memoization is much better than the worst-case complexity.

We tabulated the results for $2\leq n\leq 15$ in Table~\ref{table:1}.
A smaller version of the tables can be found
at~\cite[pg. 279]{Comtet:AdvancedCombinatorics}, and as the sequence
A295193 at OEIS~\cite{OEIS}.  The individual sequences for
$m=1,\ldots,6$ in Table~\ref{table:1} are respectively the sequences
A001147, A001205, A002829, A005815, A338978 and A339847 at OEIS.

% Table 1 and Table 2

\begin{table}[t]
  \footnotesize
  \begin{tabular}{|r|l|l|l|l|l|}\hline
    $n$ & $m=1$ & $m=2$ & $m=3$ & $m=4$ & $m=5$\\\hline\hline
      2 & 1       &0 & 0 & 0 & 0 \\ \hline
      3 & 0       &1 & 0 & 0 & 0 \\ \hline
      4 & 3       &3 & 1 & 0 & 0 \\ \hline
      5 & 0       &12 & 0 & 1 & 0 \\ \hline
      6 & 15      &70 & 70 & 15 & 1 \\ \hline
      7 & 0       &465 & 0 & 465 & 0 \\ \hline
      8 & 105     &3507 & 19355 & 19355 & 3507 \\ \hline
      9 & 0       &30016 & 0 & 1024380 & 0 \\ \hline
     10 & 945     &286884 & 11180820 & 66462606 & 66462606 \\ \hline
     11 & 0       &3026655 & 0 & 5188453830 & 0 \\ \hline
     12 & 10395   &34944085 & 11555272575 & 480413921130 & 2977635137862 \\ \hline
     13 & 0       &438263364 & 0 & 52113376310985 & 0 \\ \hline
     14 & 135135  &5933502822 & 19506631814670 & 6551246596501035 & 283097260184159421 \\ \hline
     15 & 0       &86248951243 & 0 & 945313907253606891 & 0 \\ \hline
  \end{tabular}
  
  \begin{tabular}{|r|l|l|l|l|}\hline
    $n$ & $m=6$ & $m=7$ & $m=8$ \\\hline\hline
    2 & 0 & 0 & 0 \\ \hline
    3 & 0 & 0 & 0 \\ \hline
    4 & 0 & 0 & 0 \\ \hline
    5 & 0 & 0 & 0 \\ \hline
    6 & 0 & 0 & 0 \\ \hline
    7 & 1 & 0 & 0 \\ \hline
    8 & 105 & 1 & 0 \\ \hline
    9 & 30016 & 0 & 1 \\ \hline
    10 & 11180820 & 286884 & 945 \\ \hline
    11 & 5188453830 & 0 & 3026655 \\ \hline
    12 & 2977635137862 & 480413921130 & 11555272575 \\ \hline
    13 & 2099132870973600 & 0 & 52113376310985 \\ \hline
    14 & 1803595358964773088 & 1803595358964773088 & 283097260184159421 \\ \hline
    15 & 1872726690127181663775 & 0 & 1872726690127181663775 \\ \hline
    %16 & 2329676580698022197516875 & 15138592322753242235338875 & 15138592322753242235338875 \\ \hline
    % 17 & 3443086402825299720403673760 & 0 & 149390880973211821194044293500 \\ \hline
    % 18 & 5997229769947050271535917422040 & 271849772205948458085090804526392 & 1793196665025885172290508971592750 \\ \hline
% 19 & 12218901113752712984458458475480428 & 0 & 26051341898387300707368587445031827810 \\ \hline
% 20 & 28916028252717782814901370042123535900 & 9883018890803233316233360724489799227748 & 455473295761288603097446621939956078133650 \\ \hline
% 21 & 78964245348871129554207357675851322040800 & 0 & 9526590903662623496186073183718443028656311550 \\ \hline
% 22 & 247326784004721245384687960385377976501127680 & 689121157937951859333538097288863665976145304960 & 236946226536588577932827562029441999990638366002150 \\ \hline
% 23 & 883513779022778121434134182300682844206856921535 & 0 & 6966943197297961262303012712858659986651122779167541730 \\ \hline
% 24 & 3580833804060595265972724030024883605216183845836071 & 87732200455287126135563460209581592241959402716247166591 & 240790375061694985654600350990603386365295329818412772266362 \\ \hline
% 25 & 16385912180512407636099645490806488563160196281814916000 & 0 & 9728927362650915085602119026043885987900482051738595530480763625 \\ \hline    
  \end{tabular}
  \normalsize
  \vspace{2mm}  
  \caption{The number of labeled $m$-regular graphs on $n$-vertices
    for $m=1,\ldots,8$.}\label{table:1}
\end{table}

\subsection{Enumerating graphs that realize the same degree sequences as
  binary trees}

Any binary tree with $k+1$-leaves will have $k-1$ internal vertices of
degree 3.  Thus any such tree has the degree sequence
\begin{equation}
  \label{eq:5}
  (\underbrace{3,\ldots,3}_{\text{$k-1$-times}},\underbrace{1,\ldots,1}_{\text{$k+1$-times}}).
\end{equation}
Now, let
\begin{equation}
  \label{eq:6}
  T(k) = \binom{2k}{k-1} C((\underbrace{3,\ldots,3}_{\text{$k-1$-times}},\underbrace{1,\ldots,1}_{\text{$k+1$-times}})).
\end{equation}
be the number of labeled graphs that has the same degree sequence
given in~\eqref{eq:5}.  Notice that we put a correction factor
$\binom{2k}{k-1}$ since in the original enumeration $C(d)$ vertices
are not allowed to change degree.  In the case of regular graphs, one
does not need a correction factor since every vertex has the same
degree.

Using our implementation of the enumeration algorithm we calculated
these numbers on the same setup we described above.  The results are
calculated almost immediately and they are given in
Table~\ref{table:3}.

% Table 3

\begin{table}[h]
  \footnotesize
  \begin{tabular}{{|r|l|}}\hline
    $k$ & $R(k)$ \\\hline\hline
1 & 1 \\\hline
2 & 4 \\\hline
3 & 90 \\\hline
4 & 8400 \\\hline
5 & 1426950 \\\hline
6 & 366153480 \\\hline
7 & 134292027870 \\\hline
8 & 67095690261600 \\\hline
9 & 43893900947947050 \\\hline
10 & 36441011093916429000 \\\hline
11 & 37446160423265535041100 \\\hline
12 & 46669357647008722700474400 \\\hline
13 & 69367722399061403579194432500 \\\hline
14 & 121238024532751529573125745790000 \\\hline
15 & 246171692450596203263023527657431250 \\\hline
% 16 & 574696885550555548873614279418020360000 \\\hline
% 17 & 1528648357350537361575196317338354407383750 \\\hline
% 18 & 4596088010022853824002605075300771574285175000 \\\hline
% 19 & 15510323766640868488575644570711551180769494387500 \\\hline
% 20 & 58381695150853025948853132922461121565896784647500000 \\\hline
% 21 & 243729501193016348972050432369003734301926441635587012500 \\\hline
% 22 & 1122811743410237376894084109366208508515488834433784518550000 \\\hline
% 23 & 5681608398544840648585286854894746662893383437294964982096687500 \\\hline
% 24 & 31446900965455411829754865710791710580022540563740655954757309000000 \\\hline
% 25 & 189652737607011691689037093453195592142251534661561948356227925376562500 \\\hline
  \end{tabular}
  \normalsize
  \vspace{2mm}
  \caption{The number of labeled graphs that realize the same degree
    sequence as any binary tree on $2k$ vertices.}\label{table:3}
\end{table}

In~\cite[pg.6]{Moon70} the number of labeled trees on $n$ vertices
that realize a fixed degree sequence $(d_1,\ldots,d_n)$ is calculated
as
\begin{equation}
  \label{eq:11}
  \binom{n-2}{d_1-1,\ldots,d_n-1}
\end{equation}
which is different than our calculations for the case
$(d_1,\ldots,d_n)$ given as~\eqref{eq:5}.  But note that Moon's
formula enumerates only trees that realize a particular degree
sequence while we count all graphs.

\subsection{Enumerating graphs that realize the same degree sequences as
  complete bipartite graphs}

We fix two positive integers $n\leq m$.  A complete bipartite graph
$K_{n,m}$ contains $n+m$ vertices which is split into two disjoint
sets, say black and white.  Black vertices are connected to every
white vertex and vice versa, but vertices of the same color are not
connected.  Any such graph would have the degree sequence
\begin{equation}
  \label{eq:7}
  (\underbrace{m,\ldots,m}_{\text{$n$-times}},\underbrace{n,\ldots,n}_{\text{$m$-times}})
\end{equation}
Let us write
\begin{equation}
  \label{eq:8}
  K(n,m) =
  \begin{cases}
    \binom{n+m}{n} C((\underbrace{m,\ldots,m}_{\text{$n$-times}},\underbrace{n,\ldots,n}_{\text{$m$-times}})) & \text{ if } n\neq m\\
    C((\underbrace{n,\ldots,n}_{\text{$2n$-times}})) & \text{ if } n=m
  \end{cases}
\end{equation}
for the number of labeled graphs that realize the degree sequence
given in~\eqref{eq:7} for every $m\geq 2$ and $1\leq n\leq m$.  We
tabulated the results for $2\leq m\leq 10$ and $2\leq n\leq 6$ in
Table~\ref{table:4}.  The first column of Table~\ref{table:4} is
A002061 at OEIS.

\begin{table}
  \centering
  \footnotesize
  \begin{tabular}{{|r|l|l|l|l|l|l|l|}}\hline
    $m$ & $n=2$ & $n=3$ & $n=4$ & $n=5$ & $n=6$ \\\hline
    2 & 3 & & & & \\ \hline
    3 & 7 & 70 & & & \\ \hline
    4 & 13 & 553 & 19355 & & \\ \hline
    5 & 21 & 3211 & 527481 & 66462606 & \\ \hline
    6 & 31 & 13621 & 10649191 & 6445097701 & 2977635137862 \\ \hline
    7 & 43 & 44962 & 153984573 & 466128461506 & 1051046246482968 \\ \hline
    8 & 57 & 123145 & 1601363093 & 24363074013321 & 277358348828368109 \\ \hline
    9 & 73 & 293293 & 12389057785 & 905113150135831 & 53355534127828683775 \\ \hline
    10 & 91 & 627571 & 74598011761 & 23985623638038361 & 7334781492338569314961 \\ \hline    
  \end{tabular}
  \normalsize
  \vspace{2mm}
  \caption{The number of labeled graphs that realize the same degree
    sequence as the complete bipartite graph
    $K_{n,m}$.}\label{table:4}
\end{table}

\section*{Appendix: The Code}

Since or implementation is simple and short, we opted to list the code
we used to make our calculations here in an Appendix in
Figure~\ref{fig:1}.  This way, our results can be reproduced and
verified.

We implemented our enumeration using Common Lisp~\cite{Steel:CLTL} and
a suitable memoization to control the depth of the recursive
calls. However, due to efficiency issues of the data structures we
use, our degree sequences are non-decreasing instead of being
non-increasing.  We used SBCL version 2.0.11 to run the lisp
code~\cite{SBCL}.

Our enumeration calculation requires us to calculate a finite number
of shorter degree sequences in each call. When we employ memoization,
we use a global table of already calculated results. If an enumeration
on a shorter degree sequence is needed, and if the result is already
calculated for another branch of the recursive call we recall the
result instead of calculating it from scratch.

\begin{figure}[h]
\begin{lstlisting}
(defun subsets (k xs)
   (cond ((null xs) 'nil)
         ((= 1 k) (loop for x in xs collect (list x)))
         (t (union (subsets k (cdr xs))
                   (mapcar (lambda (x) (cons (car xs) x))
                           (subsets (1- k) (cdr xs)))))))
\end{lstlisting}
\begin{lstlisting}
(defun new-degree-sequence (ds is)
  (let ((ys (copy-list (cdr ds))))
    (dolist (i is) (decf (nth i ys)))
    (delete 0 (sort ys #'<))))  
\end{lstlisting}
\begin{lstlisting}
(let ((table (make-hash-table :test #'equal)))
  (defun graph-count (ds)
    (cond 
       ((null ds) 1)
       ((oddp (reduce #'+ ds)) 0)
       (t (or (gethash ds table)
              (setf (gethash ds table)
                    (let* ((index-set (loop for i from 0 below (1- (length ds))
                                           collect i))
                           (all-subsets (subsets (car ds) index-set)))
                       (loop for xs in all-subsets sum
                            (graph-count (new-degree-sequence ds xs))))))))))
\end{lstlisting}
  \caption{Common Lisp implementation of the enumeration algorithm
    given in Theorem~\ref{thm:main}.}\label{fig:1}
\end{figure}

%\bibliographystyle{plain}
%\bibliography{references}

\end{document}